\documentclass[11pt]{amsart}

\usepackage{amsmath, amsthm, amssymb, amsfonts, mathtools, mathrsfs}
\usepackage{hyperref}
\usepackage[backend = biber, style = numeric, sorting = anyt]{biblatex}

\AtEveryBibitem{%
  \clearfield{doi}%
  \clearfield{isbn}%
  \clearfield{issn}%
  \clearfield{url}%
  \clearlist{language}%
}
\DeclareFieldFormat{title}{#1}
\DeclareFieldFormat{journaltitle}{#1}
\DeclareFieldFormat{booktitle}{#1}
\renewbibmacro{in:}{}

\theoremstyle{plain}
\newtheorem{theorem}{Theorem}[section]
\newtheorem{lemma}[theorem]{Lemma}
\newtheorem{corollary}[theorem]{Corollary}
\newtheorem{proposition}[theorem]{Proposition}

\theoremstyle{definition}
\newtheorem{example}[theorem]{Example}

\theoremstyle{remark}
\newtheorem{remark}[theorem]{Remark}

\numberwithin{equation}{section}

\title[Cohomological Milnor Fibrations]{Cohomological Milnor Fibrations of Landau--Ginzburg Models}

\author{Huijun Fan}
\address{School of Mathematics and Statistics, Wuhan University, Wuhan, Hubei, China}
\email{fanhuijun@whu.edu.cn, fanhj@math.pku.edu.cn}
\thanks{supported by National Key R\&D Program of China (Grant No. 2020YFE0204200), ZYYJQT202600011, NSFC(11831017,11890660, 11890661) and NSFC-RSF 1201101428, NSFC-RFBR 11661131005.}
\author{Jieran Wang}
\address{School of Mathematics, Sun Yat-sen University, Guangzhou, Guangdong, China}
\email{wangjr88@mail.sysu.edu.cn}
\date{\today}

\keywords{Landau--Ginzburg model; Milnor number; tameness; Lefschethz thimble; Milnor fibration; Gauss--Manin connection}
\subjclass[2020]{32S25; 32S40; 14N35; 32Q15}

\begin{document}

\begin{abstract}
    We study the global topology of \emph{fundamentally tame} Landau--Ginzburg (LG) models on complete K\"ahler manifolds.
    We first show that the superpotential $f$ is surjective onto $\mathbb{C}$: in the non-proper holomorphic setting this is far from automatic, and here it is deduced from fundamental tameness together with the completeness of $M$.
    We then prove that $f$ restricts to a locally trivial fibration over the set of its regular values, which yields a global generalization of the classical Milnor fibration and provides an Ehresmann-type fibration in the non-proper setting.
    Assuming that $f$ has finitely many critical points, we use a Morse-theoretic argument to show that the relative homology of $\bigl(M,f^{-1}(c)\bigr)$, for a regular value $c$, is concentrated in the middle degree and has rank equal to the global Milnor number $\mu_f$.
    We further identify the relative homotopy type of the fibers with that of the sublevel sets of $f$ at infinity.
    Finally, for deformations satisfying a uniform tameness condition on a Stein manifold, we construct the relative cohomology bundles at positive and negative infinity, each equipped with a natural flat Gauss--Manin connection.
    Their direct sum is called the \emph{cohomological Milnor bundle}.
\end{abstract}

\maketitle

\section{Introduction}

Landau--Ginzburg (LG) B-models and $tt^*$-geometry occupy a central place in two dimensional topological field theory and mirror symmetry, which have tight connection to the singularity theory. The singularity theory studies the local behavior of a holomorphic function at an isolated singularity. The basic invariants of the singularity are supplied by the Milnor fibration together with the Milnor number \cite{milnor1968singular}, and on its universal deformation space, K. Saito has  created the theory of primitive forms and buit the Frobenius manifold structure\cite{Saito1981Primitive,SaitoTakahashi2008,Hertling2003} by the Hodge-theoretic description through the Brieskorn lattice \cite{saito1989brieskorn}. On the other hand, the topological-field-theoretic background was provided by \cite{CecottiVafa1991,Dubrovin1993,Dubrovin1996}. The singularity theory gives a local theory of a holomorphic function, while LG models provide a global picture. The analytic and Hodge-theoretic framework of LG models was firstly developed by the first author \cite{fan2011,FanFang2021} by considering the twisted Cauchy--Riemann operators of LG models and their deformation. The final aim of this approach is to construct the LG B model quantization theory via analytical techniques and prove the mirror equivalence to the already known A-theory, FJRW theory (see \cite{FanJarvisRuan2013}). Afterwards, a series of work has been done (see  \cite{LI2022108165} and Wen \cite{Wen2016,WenFan2015}) along this way. A natural question is to study the topological informatioin of a LG model and develop the corresponding theory with respect to the local singularity theory. 

The passage from the local picture to the global one is, however, not automatic, and two difficulties arise.
First, it is not evident whether $f$ is surjective onto $\mathbb{C}$.
Second, the classical proof of the Milnor fibration over a disk invokes the Ehresmann fibration theorem, whose argument requires the map to be proper; without properness the flow that lifts a vector field from the base need not be complete, and the trivialization argument breaks down.
For polynomial functions the corresponding global theory was developed by Broughton \cite{Broughton1988}, where tameness of the polynomial controls the behavior at infinity; see also \cite{arnold2012singularities,Batyrev1993,dimca1992singularities,Kulikov_1998,LeDungTrang1977} for the singularity-theoretic background and \cite{Dubrovin1990ModernIII} for the homology theory involved.

The purpose of this article is to establish the analogous global structure for \emph{fundamentally tame} LG models on complete K\"ahler manifolds, within the analytic framework of \cite{fan2011,LI2022108165,Wen2016}.
Throughout, $M$ is a connected, noncompact, complete K\"ahler manifold of complex dimension $n$, and $f:M\to\mathbb{C}$ is a nonconstant holomorphic function.
It is fundamental tameness, together with the completeness of $M$, that replaces properness in the two obstructions above.
In the form in which it is used below, fundamental tameness provides a compact subset $K\subset M$ and a constant $C>0$ such that
\begin{equation*}
    \|\nabla f(p)\| > C \qquad\text{for all } p\in M\setminus K,
\end{equation*}
so that all critical points of $f$ lie in the compact set $K$, while the behavior of $f$ away from $K$ is controlled by a uniform gradient bound.

Our first result is a \emph{range theorem} (Theorem~\ref{thm:range}): if $(M,f)$ is a fundamentally tame LG model, then $f(M)=\mathbb{C}$.
Its proof is elementary, yet it already exhibits the mechanism that will be used throughout.
Writing $u=\operatorname{Re} f$ and $v=\operatorname{Im} f$, one shows that $u$ is unbounded both above and below by following the negative gradient flow of $u$: as the flow cannot remain in any compact set, completeness of $M$ forces its length to tend to infinity, while the gradient bound forces $u$ to decrease at a uniform rate.
The same argument applied to the level sets of $u$ then shows that $v$ is unbounded along $u^{-1}(\alpha)$ for every $\alpha\in\mathbb{R}$, which yields the surjectivity of $f$.

Our second result is a \emph{fibration theorem} (Proposition~\ref{prop:fibration}): $f$ restricts to a locally trivial fibration $f^{-1}\bigl(f(M)\setminus D\bigr)\to f(M)\setminus D$ over the set of its regular values, where $D$ denotes the set of critical values of $f$.
The argument follows the strategy of the Ehresmann fibration theorem and is inspired by \cite[Lemma 4.3]{Broughton1988}; since $f$ is not assumed to be proper, the lifted vector field is constructed explicitly, and fundamental tameness together with the completeness of $M$ is used to guarantee global existence of its flow.
This yields a global generalization of the classical Milnor fibration and provides an Ehresmann-type fibration in the non-proper setting.

Our third result computes the relative homology of the smooth fibers (Theorem~\ref{thm:main}): if $(M,f)$ is a fundamentally tame LG model whose superpotential has finitely many critical points, then for every regular value $c$ of $f$ the relative homology groups $H_*\bigl(M,f^{-1}(c)\bigr)$ are concentrated in degree $n$, with
\begin{equation*}
    H_n\bigl(M,f^{-1}(c)\bigr)\;\cong\;\mathbb{Z}^{\mu_f},
\end{equation*}
where $\mu_f=\sum_{p\in\operatorname{Crit}(f)}\mu_f(p)$ is the global Milnor number of $f$.
The proof adapts Milnor's classical Morse-theoretic argument \cite{milnor1963morse} to the non-proper, noncompact situation at hand: the auxiliary function $F_a=|f_a-c|^2$ is a nonnegative function whose zero locus is $f^{-1}(c)$ and which, outside a neighborhood of this zero locus, has exactly $\mu_f$ nondegenerate critical points, all of Morse index $n$, while compactness of the sublevel sets is replaced by the completeness of $M$ together with tameness of $F_a$ away from that neighborhood, in the spirit of Broughton \cite{Broughton1988}.
We also clarify the relation between regular fibers and the sublevel sets of $f$ at infinity (Proposition~\ref{prop:homotopy}): for $|\operatorname{Re} c|$ large, the pair $\bigl(M,f^{-1}(c)\bigr)$ is homotopy equivalent to the pair formed by $M$ and the corresponding sublevel set, and the common relative homotopy type is independent of $c$.

Our fourth and last result concerns deformations of the superpotential.
For a deformation $F(z,t)=f(z)+\sum_{i=1}^{m}t^i g_i(z)$ over a bounded parameter domain $S\subset \mathbb{C}^m$, and under a uniform tameness assumption (Theorem~\ref{thm:deformation}), the level set $E_c=\{F=c\}$ is a smooth hypersurface of $M\times S$ for $|\operatorname{Re} c|$ sufficiently large, and the projection $\pi:E_c\to S$ is a locally trivial smooth fibration.
If $M$ is Stein, combining this with the homotopy type at infinity yields (Theorem~\ref{thm:bundle}) the positive and negative relative cohomology bundles $\mathcal{H}^{+}$ and $\mathcal{H}^{-}$, each of rank $\mu$ and equipped with a natural flat Gauss--Manin connection.
We call their direct sum $\mathcal{H}_{\mathrm{tot}}=\mathcal{H}^{+}\oplus\mathcal{H}^{-}$ the \emph{cohomological Milnor bundle} associated with the deformation $F$.
In contrast with the classical Picard--Lefschetz picture, in which the Lefschetz thimbles of a Morse function form a basis of the relative homology, the construction above does not require $f_t$ to be Morse, and both bundles and their flat connections are defined over the whole deformation parameter space.

The article is organized as follows.
Section~\ref{sec:range} proves the range theorem, together with the auxiliary lemma on maximal integral curves of $-\nabla u$ that the argument requires, and records an application to holomorphic embeddings and polynomial superpotentials.
Section~\ref{sec:fibration} establishes the fibration theorem.
Section~\ref{sec:relative} computes the relative homology of the smooth fibers, and Section~\ref{sec:construction} constructs the cohomological Milnor bundles and relates the regular fibers to the sublevel sets of $f$ at infinity.

\section{Range Theorem}\label{sec:range}

As a preparation for our construction, we establish some basic properties of fundamentally tame LG models, beginning with the range theorem for $f$.
This result will later allow us to identify the relative homotopy type of the sublevel sets of $f$ with that of the fibers of $f$ when the real part of $f$ is sufficiently large.

\begin{theorem}\label{thm:range}
    If $(M, f)$ is a fundamentally tame LG model, then $f(M) = \mathbb{C}$.
\end{theorem}

\begin{proof}
    Let $u = \operatorname{Re} f$ and $v = \operatorname{Im} f$.
    Then $u$ and $v$ are nonconstant harmonic functions on $M$.

    \medskip
    \noindent\emph{Step 1: $u(M) = \mathbb{R}$.}

    Since $M$ is connected, the image $u(M)$ is an interval.
    To prove $u(M) = \mathbb{R}$, it therefore suffices to show that
    \begin{equation}
        \inf_{p \in M} u(p) = -\infty \quad \text{and} \quad \sup_{p \in M} u(p) = +\infty.
    \end{equation}
    We prove the statement for the infimum; the argument for the supremum is analogous.

    Assume for contradiction that
    \begin{equation}
        \alpha_m := \inf_{p \in M} u(p) > -\infty.
    \end{equation}

    Since $f$ is fundamentally tame, there exist a compact subset $K \subset M$ and a positive constant $C$ such that
    \begin{equation}
        \|\nabla f(p)\| > C \quad \text{for all } p \in M \setminus K.
    \end{equation}
    In particular, the critical locus of $f$ is contained in $K$.

    Let
    \begin{equation}
        \beta_m := \min_{p \in K} u(p).
    \end{equation}

    By the maximum principle, $u$ does not attain its minimum on $M$.
    Hence
    \begin{equation}
        \alpha_m < \beta_m,
    \end{equation}
    and consequently
    \begin{equation}
        (\alpha_m, \beta_m) \subset u(M).
    \end{equation}
    Moreover,
    \begin{equation}
        u^{-1}(\alpha_m, \beta_m) \subset M \setminus K.
    \end{equation}

    Choose $p \in M$ such that $u(p) \in (\alpha_m, \beta_m)$ and consider the negative gradient flow
    \begin{equation}
        \begin{cases}
            \phi'(t) = -\nabla u(\phi(t)),\\
            \phi(0) = p.
        \end{cases}
    \end{equation}

    Let $d$ denote the distance function on $M$.
    For $t \geq 0$ we have $u(\phi(t)) \leq u(p) < \beta_m$, hence $\phi(t) \in M \setminus K$.
    Since $\|\nabla u\| = \frac{1}{2}\|\nabla f\|$, it follows that
    \begin{equation}
        \|\nabla u(\phi(t))\| > \tfrac{1}{2} C.
    \end{equation}

    Since
    \begin{equation}
        u(p) - u(\phi(t)) = \int_0^t -\frac{\mathrm{d}}{\mathrm{d}\tau} u(\phi(\tau)) \, \mathrm{d}\tau = \int_0^t \|\nabla u(\phi(\tau))\|^2 \, \mathrm{d}\tau,
    \end{equation}
    we can estimate
    \begin{equation}\label{eqn:2.11}
        u(p) - u(\phi(t)) \geq \frac{1}{2} C \int_0^t \|\phi'(\tau)\| \, \mathrm{d}\tau \geq \frac{1}{2} C \, d(\phi(t), p),
    \end{equation}
    while on the other hand
    \begin{equation}\label{eqn:2.12}
        u(p) - u(\phi(t)) \geq \frac{1}{4} C^2 t.
    \end{equation}

    Let $[0, T)$ be the maximal forward interval of existence of $\phi(t)$.

    If $T = +\infty$, then by \eqref{eqn:2.12},
    \begin{equation}
        u(p) - u(\phi(t)) \to +\infty \quad \text{as } t \to +\infty,
    \end{equation}
    which contradicts the assumption that $\alpha_m > -\infty$.

    If $T < +\infty$, then by maximality of the interval the curve $\phi(t)$ cannot remain in any compact subset of $M$.
    Since $M$ is complete, this implies
    \begin{equation}
        d(\phi(t), p) \to \infty \quad \text{as } t \to T^-.
    \end{equation}
    By \eqref{eqn:2.11} we again obtain
    \begin{equation}
        u(p) - u(\phi(t)) \to +\infty \quad \text{as } t \to T^-,
    \end{equation}
    which is a contradiction.

    Therefore $\inf_{p \in M} u(p) = -\infty$.
    Applying the same argument to $-u$ gives $\sup_{p \in M} u(p) = +\infty$, and hence
    \begin{equation}
        u(M) = \mathbb{R}.
    \end{equation}

    The argument above in fact yields the following lemma.

    \begin{lemma}\label{lem:integral-curve}
        Let $M$ be a complete manifold.
        If $\gamma : [0, T) \to M$ is a maximal forward integral curve of $-\nabla u$ and $\|\nabla u\| > C$ on $\operatorname{Im} \gamma$ for some $C > 0$, then
        \begin{equation}
            u(\operatorname{Im} \gamma) = (-\infty, u(\gamma(0))].
        \end{equation}
    \end{lemma}

    \medskip
    \noindent\emph{Step 2: the range of $v$ on the level sets of $u$.}
    For every $\alpha \in \mathbb{R}$, we prove that
    \begin{equation}
        v\big(u^{-1}(\alpha)\big) = \mathbb{R}.
    \end{equation}

    Since Step 1 shows $u(M) = \mathbb{R}$, the level set $u^{-1}(\alpha)$ is nonempty.

    Because $f$ is holomorphic, the Cauchy--Riemann equations imply
    \begin{equation}
        \nabla v = J(\nabla u), \qquad \langle \nabla u, \nabla v \rangle = 0,
    \end{equation}
    where $J$ denotes the complex structure.
    Hence along any integral curve $\gamma(t)$ of $\pm \nabla v$,
    \begin{equation}
        \frac{\mathrm{d}}{\mathrm{d}t} u(\gamma(t)) = \langle \nabla u, \gamma'(t) \rangle = \pm \langle \nabla u, \nabla v \rangle = 0.
    \end{equation}
    Therefore every integral curve of $\pm \nabla v$ that starts in $u^{-1}(\alpha)$ remains inside this level set.

    Since the critical locus of $f$ is contained in the compact set $K$, it is a compact complex analytic space.
    Consider its connected components.
    By the maximum principle for holomorphic functions on connected complex analytic spaces \cite[Chapter 5, Section 5]{GrauertRemmert2004}, the function $f$ is constant on each connected component, so that on each such component both $u$ and $v$ are constant.
    Each component is therefore either contained in $u^{-1}(\alpha)$ or disjoint from it.

    Suppose that there are $k$ connected components of the critical locus that are contained in $u^{-1}(\alpha)$.
    If $k = 0$, then $u^{-1}(\alpha)$ contains no critical points of $f$, so we can choose $C > C_1 > 0$ such that $\|\nabla v\| > C_1$ on $M\setminus K$ and on $u^{-1}(\alpha)\cap K$, hence on $u^{-1}(\alpha)$.
    Then by Lemma~\ref{lem:integral-curve}, $v\big(u^{-1}(\alpha)\big) = \mathbb{R}$.
    Otherwise, let $k>0$ and let $S_1, \ldots, S_k$ be the connected components of the critical locus that are contained in $u^{-1}(\alpha)$, arranged so that
    \begin{equation}
        v(S_1) \leq \cdots \leq v(S_k).
    \end{equation}

    Since $f$ is a nonconstant holomorphic function, it is an open map.
    Thus there exists $\epsilon > 0$ such that for each $i$,
    \begin{equation}
        D_\epsilon\big(f(S_i)\big) \subset f(M),
    \end{equation}
    and for $i \neq j$ and $v(S_i) \neq v(S_j)$,
    \begin{equation}
        D_\epsilon\big(f(S_i)\big) \cap D_\epsilon\big(f(S_j)\big) = \emptyset.
    \end{equation}
    Hence
    \begin{equation}
        \big(v(S_i) - \epsilon, \, v(S_i) + \epsilon\big) \subset v\big(u^{-1}(\alpha)\big).
    \end{equation}

    Choose an open neighborhood $U_\alpha \subset M$ containing $S_1, \ldots, S_k$ such that
    \begin{equation}
        f(U_\alpha) \subset \bigcup_{i=1}^{k} D_{\epsilon'}\big(f(S_i)\big)
    \end{equation}
    for some $0 < \epsilon' < \epsilon$.

    Then
    \begin{equation}
        (K \setminus U_\alpha) \cap u^{-1}(\alpha)
    \end{equation}
    is compact, and there exists $\widetilde{C}_\alpha > 0$ such that
    \begin{equation}
        \|\nabla v\| > \widetilde{C}_\alpha \quad \text{on } u^{-1}(\alpha) \setminus U_\alpha.
    \end{equation}

    Choose $p_1 \in u^{-1}(\alpha) \setminus U_\alpha$ such that
    \begin{equation}
        v(p_1) \in \big(v(S_1) - \epsilon, \, v(S_1) - \epsilon'\big).
    \end{equation}

    Consider the maximal forward integral curve $\gamma_0$ of $-\nabla v$ starting at $p_1$.
    Since $v$ strictly decreases along this curve, the curve cannot intersect $U_{\alpha}$.
    Let $\Gamma_0$ be the image of $\gamma_0$.
    By Lemma~\ref{lem:integral-curve},
    \begin{equation}
        v(\Gamma_0) = (-\infty, v(p_1)].
    \end{equation}

    Set $j_1 = 1$.
    Choose $q_1 \in u^{-1}(\alpha) \setminus U_{\alpha}$ such that
    \begin{equation}
        v(q_1) \in \bigl(v(S_{j_1}) + \epsilon', \ v(S_{j_1}) + \epsilon\bigr).
    \end{equation}

    Consider the maximal forward integral curve $\gamma_1$ of $\nabla v$ starting at $q_{j_1}$.
    If $\gamma_1$ does not intersect $U_{\alpha}$, then by Lemma~\ref{lem:integral-curve}
    \begin{equation}
        v(\gamma_1) = [v(q_{j_1}), +\infty),
    \end{equation}
    and the argument terminates.

    Otherwise, there exist $t_1$ and $j_2 > j_1$ such that
    \begin{equation}
        p_{j_2} := \gamma_1(t_1) \in U_{\alpha}, \quad v(p_{j_2}) \in \bigl(v(S_{j_2}) - \epsilon', \ v(S_{j_2}) + \epsilon'\bigr).
    \end{equation}
    Let
    \begin{equation}
        \Gamma_1 = \gamma_1([0, t_1]).
    \end{equation}

    Repeating the same procedure, we obtain indices
    \begin{equation}
        1 = j_1 < j_2 < \cdots < j_m \le k
    \end{equation}
    for some $1 \le m \le k$, points $p_{j_2}, \ldots, p_{j_m} \in U_{\alpha}$ such that
    \begin{equation}
        v(p_{j_i}) \in \bigl(v(S_{j_i}) - \epsilon', \ v(S_{j_i}) + \epsilon'\bigr),
    \end{equation}
    points $q_{j_1}, \ldots, q_{j_m} \in u^{-1}(\alpha) \setminus U_{\alpha}$ such that
    \begin{equation}
        v(q_{j_i}) \in \bigl(v(S_{j_i}) + \epsilon', \ v(S_{j_i}) + \epsilon\bigr),
    \end{equation}
    and curve segments $\Gamma_1, \ldots, \Gamma_m$, where $\Gamma_k$ connects $q_{j_k}$ to $p_{j_{k+1}}$ for $k = 1, \ldots, m-1$ and $\Gamma_m$ is the image of a maximal forward integral curve of $\nabla v$.
    By Lemma~\ref{lem:integral-curve},
    \begin{equation}
        v(\Gamma_m) = [v(q_{j_m}), +\infty).
    \end{equation}

    The sets
    \begin{equation}
        v(\Gamma_0), \ v(\Gamma_1), \ldots, \ v(\Gamma_m)
    \end{equation}
    are pairwise disjoint closed intervals whose complement in $\mathbb{R}$ is
    \begin{equation}
        \bigl(v(p_{j_1}), v(q_{j_1})\bigr) \cup \cdots \cup \bigl(v(p_{j_m}), v(q_{j_m})\bigr).
    \end{equation}

    Each of these complementary intervals is contained in one of the open intervals
    \begin{equation}
        \bigl(v(S_i) - \epsilon, \ v(S_i) + \epsilon\bigr),
    \end{equation}
    which in turn lie in $v\bigl(u^{-1}(\alpha)\bigr)$.
    Therefore
    \begin{equation}
        v\bigl(u^{-1}(\alpha)\bigr) = \mathbb{R}.
    \end{equation}

    Hence $f(M) = \mathbb{C}$.
\end{proof}

As an application of the range theorem, we record the following corollary concerning holomorphic embeddings and polynomial superpotentials.

\begin{corollary}\label{cor:embedding}
    Let $(M,f)$ be a K\"ahlerian LG model, and let $\psi: N\hookrightarrow M$ be a holomorphic embedding.
    If the induced K\"ahlerian LG model $(N, \psi^* f)$ is fundamentally tame, then $f(M)=\mathbb{C}$.
    In particular, the following two cases hold:
    \begin{enumerate}
        \item If $f$ is a nonconstant polynomial defined on the Euclidean space $\mathbb{C}^N$, then $f(\mathbb{C}^N)=\mathbb{C}$.
        \item If $f$ is a Laurent polynomial defined on $\bigl((\mathbb{C}^*)^N, \sum_{i=1}^N |dz_i|^2/|z_i|^2\bigr)$ having one monomial with positive degree and one monomial with negative degree, then $f\bigl((\mathbb{C}^*)^N\bigr)=\mathbb{C}$.
    \end{enumerate}
\end{corollary}

\begin{proof}
    Since $(N, \psi^* f)$ is fundamentally tame, the range theorem (Theorem~\ref{thm:range}) applied to this LG model gives $\psi^*f(N)=\mathbb{C}$, and hence $f(M)=\mathbb{C}$.
    For the polynomial case, choose $v\in\mathbb{C}^N$ such that $z\mapsto f(zv)$ is nonconstant.
    This one-variable polynomial is surjective by the fundamental theorem of algebra, so $f(\mathbb{C}^N)=\mathbb{C}$.
    For the Laurent case, choose $\alpha\in(\mathbb{C}^*)^N$ generically so that $h(z)=f(\alpha_1z,\ldots,\alpha_Nz)$ retains a positive and a negative power of $z$.
    Such a choice is possible because the coefficient of each total degree is a nonzero Laurent polynomial in $\alpha$.
    If $-r<0$ is the lowest exponent of $h$, then for every $c\in\mathbb{C}$, $z^r(h(z)-c)$ is a nonconstant polynomial with nonzero constant term.
    It therefore has a nonzero root, which gives $h(z)=c$ and proves the second case.
\end{proof}

The following class of quiver superpotentials gives examples satisfying the hypothesis of Corollary~\ref{cor:embedding}~(2).

\begin{example}\label{exa:quiver}
    Let $Q=(Q_0,Q_1)$ be a finite directed quiver with $|Q_0|\ge2$, vertex set $Q_0$, and arrow set $Q_1$.
    Attach a nonzero variable $y_v\in\mathbb{C}^*$ to every vertex $v\in Q_0$, and to an arrow $a:u\to v$ associate the character
    \begin{equation*}
        \chi_a=\frac{y_v}{y_u}.
    \end{equation*}
    The associated quiver superpotential is the Laurent polynomial
    \begin{equation*}
        W_Q=\sum_{a\in Q_1}c_a\,\chi_a
        =\sum_{a\in Q_1}c_a\,\frac{y_{h(a)}}{y_{t(a)}},
        \qquad c_a\in\mathbb{C}^*,
    \end{equation*}
    where $t(a)$ and $h(a)$ denote the tail and the head of the arrow $a$.
    Since the simultaneous rescaling $y_v\mapsto\lambda y_v$ leaves every ratio $y_{h(a)}/y_{t(a)}$ invariant, the natural habitat of $W_Q$ is the algebraic torus
    \begin{equation*}
        T_Q=(\mathbb{C}^*)^{Q_0}/\mathbb{C}^*\simeq(\mathbb{C}^*)^{|Q_0|-1},
    \end{equation*}
    where we choose coordinates by fixing a base vertex $v_0\in Q_0$ and setting $y_{v_0}=1$.
    We equip $T_Q$ with the standard logarithmic metric $\sum_i |dy_i|^2/|y_i|^2$.

    The exponent of the arrow $a:u\to v$ is the vector $m_a=e_v-e_u$ in the character lattice $\{m\in\mathbb{Z}^{Q_0}:\sum_{v\in Q_0}m_v=0\}$.
    After combining equal characters, the Newton polytope of $W_Q$ is
    \begin{equation*}
        \Delta_Q=\operatorname{Conv}\{m:c_m\ne0\},\qquad c_m=\sum_{a\in Q_1:\,m_a=m}c_a.
    \end{equation*}
    Assume that $\Delta_Q$ is full-dimensional and that $0\in\operatorname{Int}(\Delta_Q)$.
    Then the actual support of $W_Q$ contains exponents of both positive and negative total degree in the chosen torus coordinates, so the hypothesis of Corollary~\ref{cor:embedding}~(2) is satisfied, and we conclude that
    \begin{equation*}
        W_Q(T_Q)=\mathbb{C}.
    \end{equation*}

    For $n\ge2$ and $q\in\mathbb{C}^*$, a concrete instance is the Hori--Vafa mirror of $X=\mathbb{P}^{n-1}$, namely
    \begin{equation*}
        W=x_1+\cdots+x_{n-1}+\frac{q}{x_1\cdots x_{n-1}}
        \qquad\text{on }(\mathbb{C}^*)^{n-1},
    \end{equation*}
    whose monomials $x_i$ have positive degree and whose monomial $\frac{q}{x_1\cdots x_{n-1}}$ has negative degree; hence $W\bigl((\mathbb{C}^*)^{n-1}\bigr)=\mathbb{C}$ by Corollary~\ref{cor:embedding}~(2).
\end{example}

\section{Fibration Theorem}\label{sec:fibration}

In this section we establish a fibration theorem for the superpotential $f$.
It will be used both in the computation of relative homology and in the comparison between fibers and sublevel sets.

\begin{proposition}\label{prop:fibration}
    Let $(M,f)$ be a fundamentally tame LG model.
    Denote by $D$ the set of critical values of $f$.
    Then
    \begin{equation}
        f|_{f^{-1}(f(M)\setminus D)} : f^{-1}(f(M) \setminus D) \to f(M) \setminus D
    \end{equation}
    is a locally trivial fibration.
\end{proposition}

The argument follows the strategy of the Ehresmann fibration theorem and is inspired by \cite[Lemma 4.3]{Broughton1988}.
Since $f$ is a submersion over $f(M) \setminus D$, a trivialization is obtained by lifting a vector field from the base and constructing an isotopy between fibers along its flow.
As $f$ is not assumed to be proper, completeness of the lifted flow is not automatic; we therefore construct the lifted vector field explicitly, and use the fundamentally tame condition together with the completeness of $M$ to ensure that its flow exists globally.

\begin{proof}
    Let $U \subset f(M)$ be a convex subset with $\overline{U} \subset f(M) \setminus D$.
    Since $f(M) \setminus D$ is open in $\mathbb{C}$, we may choose an open neighborhood $V$ of $\overline{U}$ with $\overline{V} \subset f(M) \setminus D$.
    Then $f$ is a submersion on $f^{-1}(\overline{V})$.

    We look for a vector field $\tilde{X}$ on $f^{-1}(V)$ lifting the constant vector field on the base, that is, satisfying
    \begin{equation}
        \mathrm{d}f(\tilde{X}) = 1.
    \end{equation}
    Since $f$ is holomorphic, $\partial f$ is a $(1,0)$-form, and its gradient
    \begin{equation}
        \nabla f = g^{i\bar{j}} \frac{\partial f}{\partial z^i} \frac{\partial}{\partial \bar{z}^j}
    \end{equation}
    is a $(0,1)$-vector field.
    Define
    \begin{equation}
        \tilde{X} = \frac{1}{\|\nabla f\|^2} \overline{\nabla f}.
    \end{equation}
    Then $\tilde{X}$ is a smooth $(1,0)$-vector field, and a direct computation gives
    \begin{equation}
        \tilde{X} f = \frac{1}{\|\nabla f\|^2} \overline{\nabla f}(f) = \frac{1}{\|\nabla f\|^2} \|\nabla f\|^2 = 1.
    \end{equation}
    Hence $\partial f(\tilde{X}) = 1$, so that $\tilde{X}$ lifts the constant unit vector field on the base.

    Because $f$ is fundamentally tame and has no critical point on $f^{-1}(\overline{V})$, the function $\|\nabla f\|$ admits a positive lower bound on $f^{-1}(\overline{V})$.
    In particular $\tilde{X}$ is bounded.

    Choose a smooth function $\mu$ on $M$ such that
    \begin{equation}
        \mu \equiv 1 \ \text{ on } \ f^{-1}(U), \qquad \mu \equiv 0 \ \text{ on } \ M \setminus f^{-1}(V), \qquad 0 \le \mu \le 1,
    \end{equation}
    and define a global smooth vector field on $M$ by
    \begin{equation}
        X(p) =
        \begin{cases}
            \mu(p)\tilde{X}(p), & p \in f^{-1}(V), \\
            0, & p \notin f^{-1}(V).
        \end{cases}
    \end{equation}
    Then $X$ is bounded.

    Fix $c_0 \in U$ and, for each $c \in U$, set
    \begin{equation}
        X_c = (c_0 - c)X,
    \end{equation}
    denoting the global flow of $2 \operatorname{Re} X_c$ by $\phi_c(p,t)$.
    Since $X_c$ is bounded and $M$ is complete, this flow is complete.

    Then for $p \in f^{-1}(c)$ we have
    \begin{equation}
        \frac{\mathrm{d}}{\mathrm{d}t} f(\phi_c(p,t)) = (c_0 - c)\mu(\phi_c(p,t)).
    \end{equation}
    Integrating yields
    \begin{equation}
        f(\phi_c(p,t)) = c + \int_0^t (c_0 - c)\mu(\phi_c(p,s)) \, \mathrm{d}s.
    \end{equation}

    Since $0 \le \mu \le 1$, for $t \in [0,1]$ we may write
    \begin{equation}
        f(\phi_c(p,t)) = (1 - \theta(t))c + \theta(t)c_0
    \end{equation}
    for some $\theta(t) \in [0,1]$; hence $f(\phi_c(p,t))$ lies on the segment joining $c$ and $c_0$.
    Because $U$ is convex, this segment is contained in $U$, so $\mu = 1$ along the trajectory for $t \in [0,1]$.
    Therefore
    \begin{equation}
        f(\phi_c(p,t)) = (1 - t)c + tc_0.
    \end{equation}

    It follows that
    \begin{equation}
        \phi_c(f^{-1}(c), t) \subset f^{-1}\bigl((1 - t)c + tc_0\bigr).
    \end{equation}
    Since the flow is invertible, with inverse given by $t \mapsto -t$, the reverse inclusion also holds.
    Hence
    \begin{equation}
        \phi_c(f^{-1}(c), t) = f^{-1}\bigl((1 - t)c + tc_0\bigr),
    \end{equation}
    and the induced map
    \begin{equation}
        \tilde{\phi}_{c,t} : f^{-1}(c) \to f^{-1}\bigl((1 - t)c + tc_0\bigr)
    \end{equation}
    is a diffeomorphism.

    Define
    \begin{equation}
        H(p,t) = \phi_{f(p)}(p,t), \qquad p \in f^{-1}(U), \ t \in [0,1].
    \end{equation}
    Then
    \begin{equation}
        \frac{\partial H}{\partial t} = 2\operatorname{Re}((c_0 - f(p))X)(H(p,t)), \qquad H(p,0) = p.
    \end{equation}
    Thus $H$ is smooth and
    \begin{equation}
        H(f^{-1}(U), 1) = f^{-1}(c_0).
    \end{equation}
    Hence $H$ defines a deformation retraction of $f^{-1}(U)$ onto $f^{-1}(c_0)$.

    Finally define
    \begin{equation}
        \tilde{H} : f^{-1}(U) \to f^{-1}(c_0) \times U, \qquad \tilde{H}(p) = (H(p,1), f(p)).
    \end{equation}
    Its inverse is obtained by flowing backwards along the corresponding segment in $U$.
    Therefore $\tilde{H}$ is a diffeomorphism, and $f|_{f^{-1}(U)}$ is a trivial fibration over $U$.
\end{proof}

\section{Relative Homology of Smooth Fibers}\label{sec:relative}

In this section we compute the relative homology of the regular fibers of $f$.
The resulting description of $H_*(M, f^{-1}(c))$ will be used in the next section to construct the relative cohomology bundle over the deformation parameter space.

\begin{theorem}\label{thm:main}
    Let $(M,f)$ be a fundamentally tame LG model.
    Assume that $f$ has finitely many critical points and set $\mu = \mu_f$.
    Then for any regular value $c$ of $f$ the relative homology groups $H_*(M, f^{-1}(c))$ are concentrated in degree $n$, and
    \begin{equation}
        H_n(M, f^{-1}(c)) \cong \mathbb{Z}^{\mu}.
    \end{equation}
\end{theorem}

We shall use the following lemma, which is a direct consequence of Lemma B.3 in \cite{milnor1968singular} applied to $\partial f$.

\begin{lemma}\label{lem:perturbation}
    Let $V \subset \mathbb{C}^n$ be an open neighborhood of the origin.
    Let $f : V \to \mathbb{C}$ be a holomorphic function such that the origin is the only critical point of $f$ in $V$.
    Let $U \Subset V$ be an open neighborhood of the origin.
    Then there exists $\epsilon > 0$ such that for almost all tuples $a = (a_1, \ldots, a_n) \in \mathbb{C}^n$ satisfying $\|a\| < \epsilon$, the perturbed function
    \begin{equation}
        \tilde{f}(z) = f(z) + a_1 z_1 + \cdots + a_n z_n
    \end{equation}
    is a Morse function in $U$.
    Moreover, the number of critical points of $\tilde{f}$ equals the Milnor number $\mu_f$.
\end{lemma}

\begin{proof}[Proof of Theorem~\ref{thm:main}]
    We follow the procedure in Section 4 of \cite{Broughton1988} to construct a CW approximation for $(M, f^{-1}(c))$.

    \medskip
    \noindent\emph{Step 1.}
    Let $p_1, \ldots, p_k$ be all critical points of $f$.
    Choose pairwise disjoint coordinate charts $(V_i, (z^i_1, \ldots, z^i_n))$ centered at $p_i$ such that $\overline{V_i}$ is compact and $\overline{V_i} \cap f^{-1}(c) = \varnothing$, and open neighborhoods $U_i$ of $p_i$ such that $\overline{U_i} \subset V_i$.
    Choose smooth functions $\phi^i$ supported in $V_i$ such that $\phi^i = 1$ on $U_i$.
    For each $i$, let $a^i = (a^i_1, \ldots, a^i_n) \in \mathbb{C}^n$ be a tuple of complex numbers.
    On $V_i$ define
    \begin{equation}
        Z_i = (z^i_1 \phi^i, \ldots, z^i_n \phi^i),
    \end{equation}
    and
    \begin{equation}
        A^i_{a^i} = \biggl( \sum_{j=1}^{n} a^i_j z^i_j \biggr) \phi^i = \langle a^i, Z_i \rangle, \qquad A_a = \sum_{i=1}^{k} A^i_{a^i},
    \end{equation}
    where each $A^i_{a^i}$ is understood to be zero outside $V_i$.
    We then define the perturbed function $f_a$ by
    \begin{equation}
        f_a = f + A_a.
    \end{equation}

    Define
    \begin{equation}
        F_a = |f_a - c|^2, \qquad F = |f - c|^2.
    \end{equation}
    Clearly $F_a$ coincides with $F$ outside $\bigcup_{i=1}^{k} V_i$.
    Let
    \begin{equation*}
        \eta = \min_{\bigcup_i \overline{V_i}} |f - c| > 0.
    \end{equation*}
    Choose $a$ sufficiently small that $\|A_a\|_\infty < \eta/2$.
    Then $F_a \neq 0$ on $\bigcup_{i=1}^{k} \overline{V_i}$.
    Consequently, $F_a^{-1}(0) = F^{-1}(0) = f^{-1}(c)$.
    Let $\Delta$ be the open neighborhood of $f^{-1}(c)$ defined by $F_a < \delta^2$, where $0 < \delta < \eta/2$.
    Then $\overline{\Delta} \subset M \setminus \bigcup_{i=1}^{k} \overline{V_i}$.

    We first prove that $F_a$ is fundamentally tame on $M \setminus \Delta$ and that, if each $\|a^i\|$ is sufficiently small, $F_a$ has no critical point on $M \setminus (\Delta \cup \bigcup_{i=1}^{k} U_i)$.
    Since $F$ is real-valued, $\mathrm{d}F = 0$ is equivalent to $\partial F = 0$, and we compute
    \begin{equation}
        \partial F = \overline{(f-c)}\,\partial f.
    \end{equation}
    Since $f$ is fundamentally tame, there exists a compact subset $K \subset M$ and a positive constant $C$ such that $\|\partial f\| = \|\nabla f\| > C$ on $M \setminus K$; we may choose $K$ so that $\bigcup_{i=1}^{k} \overline{V_i} \subset K$.
    Then on the set $M \setminus (K \cup \Delta)$ we have
    \begin{equation}
        \|\partial F\| = |f-c|\,\|\partial f\| > \delta C.
    \end{equation}
    Since $F_a$ coincides with $F$ outside $K$, the inequality above holds for $\|\partial F_a\|$ as well; hence $F_a$ is fundamentally tame on $M \setminus \Delta$.

    We now consider the behavior of $F_a$ on $M \setminus \bigcup_{i=1}^{k} U_i$.
    Since $\partial f \neq 0$ and $f \neq c$ on this set, $F$ has no critical point there.
    Since $F_a$ coincides with $F$ outside $\bigcup_{i=1}^{k} \overline{V_i}$, it suffices to show that $F_a$ has no critical point on $\overline{V_i \setminus U_i}$ for each $i$.
    We argue by estimating $\|\mathrm{d}F_a - \mathrm{d}F\|$ on these sets.
    Observe
    \begin{equation}
        F_a = F + 2\operatorname{Re}\bigl( (f-c)\overline{A_a} \bigr) + |A_a|^2.
    \end{equation}
    Hence
    \begin{equation}
        \mathrm{d}F_a - \mathrm{d}F = \mathrm{d}\bigl( 2\operatorname{Re}((f-c)\overline{A_a}) \bigr) + \mathrm{d}(|A_a|^2).
    \end{equation}
    We obtain
    \begin{equation}
        \begin{aligned}
            \|\mathrm{d}F_a - \mathrm{d}F\| &= \bigl\| \mathrm{d}\bigl( 2\operatorname{Re}((f-c)\overline{A_a}) \bigr) + \mathrm{d}(|A_a|^2) \bigr\| \\
            &= \bigl\| 2\operatorname{Re}(\overline{A_a}\,\mathrm{d}f + (f-c)\,\mathrm{d}\overline{A_a}) + 2\operatorname{Re}(\overline{A_a}\,\mathrm{d}A_a) \bigr\| \\
            &\le 2 \bigl( |A_a|\,\|\mathrm{d}f\| + |f-c|\,\|\mathrm{d}A_a\| + |A_a|\,\|\mathrm{d}A_a\| \bigr) \\
            &= 2 \bigl( |\langle a^i, Z_i \rangle|\,\|\mathrm{d}f\| + |f-c|\,\|\langle a^i, \mathrm{d}Z_i \rangle\| \\
            &\qquad {}+ |\langle a^i, Z_i \rangle|\,\|\langle a^i, \mathrm{d}Z_i \rangle\| \bigr) \\
            &\le 2 \bigl( \|a^i\|\,\|Z_i\|\,\|\mathrm{d}f\| + |f-c|\,\|a^i\|\,\|\mathrm{d}Z_i\| \\
            &\qquad {}+ \|a^i\|\,\|Z_i\|\,\|a^i\|\,\|\mathrm{d}Z_i\| \bigr).
        \end{aligned}
    \end{equation}

    On the compact set $\overline{V_i \setminus U_i}$ the following functions and forms are bounded:
    \begin{equation}
        Z_i, \quad \mathrm{d}Z_i, \quad f-c, \quad \partial f.
    \end{equation}
    Hence there exists a constant $K_i > 0$ such that
    \begin{equation}
        \|\mathrm{d}F_a - \mathrm{d}F\| \le K_i \|a^i\|
    \end{equation}
    for $\|a^i\|$ sufficiently small.

    Moreover, since $\partial f$ does not vanish on $\overline{V_i \setminus U_i}$, there exists $m_i > 0$ such that
    \begin{equation}
        \|\mathrm{d}F\| = 2|f-c|\,\|\partial f\| \ge m_i
    \end{equation}
    on $\overline{V_i \setminus U_i}$.

    Choose $\epsilon > 0$ such that
    \begin{equation}
        K_i \epsilon \le \frac{m_i}{2}
    \end{equation}
    for every $i$.
    Then whenever $\|a^i\| < \epsilon$,
    \begin{equation}
        \|\mathrm{d}F_a - \mathrm{d}F\| \le \frac{1}{2}\|\mathrm{d}F\|
    \end{equation}
    on $\overline{V_i \setminus U_i}$.

    Second, by Lemma~\ref{lem:perturbation}, we may shrink $\epsilon$ so that for almost all $a^i$ with $\|a^i\| < \epsilon$, the function $f_a$ is Morse on $U_i$ and has exactly $\mu_f(p_i)$ critical points in $U_i$.
    We fix such an $a^i$ for each $i$, chosen so that $F_a$ takes distinct values at distinct critical points; this will be used in Step 2.
    Since $f_a$ is holomorphic on $U_i$ while $F_a$ is real-valued, the critical points of $F_a$ in $U_i$ are determined by
    \begin{equation}
        \partial F_a = \overline{(f_a - c)}\,\partial f_a = 0.
    \end{equation}
    Thus they coincide with the critical points of $f_a$ in $U_i$, and their number equals $\mu_f(p_i)$.

    Finally, we compute the Morse index of each critical point of $F_a$.
    Since $f_a$ is holomorphic on $U_i$, the Hessian of $F_a$ in $(z^i, \bar{z}^{\,i})$ coordinates has the form
    \begin{equation}
        \operatorname{Hess} F_a =
        \begin{pmatrix}
            \overline{(f_a - c)}\,\partial^2 f_a & (\partial f_a)(\overline{\partial f_a})^T \\
            (\overline{\partial f_a})(\partial f_a)^T & (f_a - c)\,\overline{\partial^2 f_a}
        \end{pmatrix}.
    \end{equation}
    Since $f_a$ differs from $f$ by linear terms on $U_i$, their second derivatives coincide:
    \begin{equation}
        \partial^2 f_a = \partial^2 f.
    \end{equation}
    At a critical point $p \in U_i$ we have $\partial f_a(p) = 0$, so the off-diagonal blocks of the Hessian vanish.
    Denote the nonzero constant $\overline{(f_a(p)-c)}$ by $\lambda_p$.
    Then the Hessian of $F_a$ at $p$ has the form
    \begin{equation}
        \operatorname{Hess} F_a(p) =
        \begin{pmatrix}
            \lambda_p\,\partial^2 f(p) & 0 \\
            0 & \overline{\lambda_p\,\partial^2 f(p)}
        \end{pmatrix}.
    \end{equation}
    It equals the Hessian at $p$ of the function $\lambda_p f_a + \overline{\lambda_p f_a} = 2\operatorname{Re}(\lambda_p f_a)$.
    Since $\lambda_p \neq 0$, the functions $F_a$ and $2\operatorname{Re}(f_a)$ have the same Morse index at $p$.
    By standard complex Morse theory, the real part of a holomorphic function on a complex $n$-manifold has Morse index $n$ at a nondegenerate critical point.
    Hence the Morse index of $F_a$ at $p$ equals $n$.

    \medskip
    \noindent\emph{Step 2.} We construct a CW approximation for $(M, f^{-1}(c))$ by a Morse-theoretic argument in the spirit of Milnor \cite{milnor1963morse} and Broughton \cite{Broughton1988}.

    Recall that, by the definition of $\Delta$, $\overline{\Delta} = f_a^{-1}(\overline{D_\delta(c)}) = f^{-1}(\overline{D_\delta(c)})$, where $D_\delta(c)$ denotes the open disk of radius $\delta$ centered at $c$ in $\mathbb{C}$.
    By Proposition~\ref{prop:fibration}, $f^{-1}(c)$ is a deformation retract of $\overline{\Delta}$; it therefore suffices to understand the homotopy type of $(M, \overline{\Delta})$.

    Let
    \begin{equation}
        M^r = F_a^{-1}[0, r], \qquad M^{s,r} = F_a^{-1}[s, r].
    \end{equation}
    Clearly $M^{\delta^2} = \overline{\Delta}$.

    As in Theorems 3.1 and 3.2 of \cite{milnor1963morse}, we claim:
    \begin{enumerate}
        \item If $r \ge s \ge \delta^2$ and $F_a$ has no critical point in $M^{s,r}$, then $M^s$ is a deformation retract of $M^r$.
        \item If $p$ is a critical point of index $k$ with $F_a(p) = s$, then for sufficiently small $\epsilon > 0$, $M^{s+\epsilon}$ has the homotopy type of $M^{s-\epsilon}$ with a $k$-cell attached.
    \end{enumerate}
    Then induction on the critical values of $F_a$ shows that $M$ has the homotopy type of $\overline{\Delta}$ with $\mu$ $n$-cells attached.

    In Milnor's argument the sets $M^{s,r}$ were assumed compact.
    As observed by Broughton, the argument remains valid under the weaker assumptions that
    \begin{enumerate}
        \item each $M^{s,r}$ is complete with respect to the given Riemannian metric;
        \item $\|\nabla F_a\|$ is bounded away from $0$ outside a compact neighborhood of the critical points in $M^{s,r}$.
    \end{enumerate}
    In the present situation $M$ is complete and $F_a$ is fundamentally tame on $M \setminus \Delta$, so the two conditions above are satisfied for every $M^{s,r}$ with $s \ge \delta^2$.
    Therefore the proofs of Theorems 3.1 and 3.2 in \cite{milnor1963morse} apply without modification, and our claims hold.

    We conclude that $M$ has the homotopy type of $f^{-1}(c)$ with $\mu$ $n$-cells attached.
    Consequently,
    \begin{equation}
        H_k(M, f^{-1}(c)) = 0 \quad \text{for } k \neq n,
    \end{equation}
    and
    \begin{equation}
        H_n(M, f^{-1}(c)) \cong \mathbb{Z}^{\mu}.
    \end{equation}
\end{proof}

\begin{remark}\label{rem:broughton}
    When $M = \mathbb{C}^{n+1}$ and $f$ is a tame polynomial in the sense of Broughton, the statement above reduces to a special case of \cite[Theorem 1.2]{Broughton1988}, where only regular fibers are considered.
    Since $\mathbb{C}^{n+1}$ is contractible, the homology of a regular fiber can be recovered from the relative CW approximation of the pair $(M, f^{-1}(c))$ obtained from the Morse-theoretic construction.

    In the present setting, however, $M$ is a complete K\"ahler manifold and need not be contractible.
    What one obtains is therefore the relative homology of the pair $(M, f^{-1}(c))$ rather than the homology of the fiber itself.
\end{remark}

\section{Construction of Cohomological Milnor Fibrations}\label{sec:construction}

In this section we study deformations of the superpotential $f$ and the fibration structure they induce.
Using a fibration theorem analogous to Proposition~\ref{prop:fibration}, we show that nearby fibers are related by isotopies in $M$ depending smoothly on the deformation parameter.
It follows that the relative cohomology groups $H^n(M, f_t^{-1}(c))$ form a smooth vector bundle over the parameter space, equipped with a natural flat connection.
To relate this construction to the Lefschetz thimbles of Picard--Lefschetz theory, we also identify the relative homotopy type of the fibers with that of the corresponding sublevel sets at infinity.

Let $S \subset \mathbb{C}^m$ be a bounded open neighborhood of $0$.
We consider deformations of $f$ of the form
\begin{equation}
    F(z,t) = f(z) + \sum_{i=1}^{m} t^i g_i(z), \quad (z,t) \in M \times S,
\end{equation}
where the $g_i$ are holomorphic functions on $M$ and $g_1 = 1$.

For each fixed $t \in S$, denote
\begin{equation}
    f_t(z) = F(z,t).
\end{equation}

\begin{theorem}\label{thm:deformation}
    Suppose that there exist a compact subset $K \subset M$ and constants $C_i > 0$ for $i = 1, \ldots, m$ such that for every $t \in S$,
    \begin{equation}
        \|\nabla f_t(z)\| > C_i |g_i(z)| \quad \text{for all } z \in M \setminus K.
    \end{equation}
    Then for $c \in \mathbb{C}$ with $|\operatorname{Re} c|$ sufficiently large, the level set
    \begin{equation}
        E_c = \{(z,t) \in M \times S \mid F(z,t) = c\}
    \end{equation}
    is a smooth hypersurface of $M \times S$, and the projection
    \begin{equation}
        \pi : E_c \to S
    \end{equation}
    is a locally trivial smooth fibration.
\end{theorem}

\begin{proof}
    First, we prove that $E_c$ is a smooth hypersurface of $M \times S$.

    Taking $i = 1$ in the assumption, we obtain $\|\nabla f_t(z)\| > C_1$ for all $z \in M \setminus K$ and all $t \in S$.
    Hence all critical points of $f_t$ are contained in $K$, uniformly for $t \in S$.
    Moreover, since $S$ is bounded, the values of $f_t$ on $K$ are uniformly bounded for $t \in S$; choose $R > 0$ such that
    \begin{equation}
        |\operatorname{Re} f_t(z)| \le R \quad \text{for all } z \in K \text{ and all } t \in S.
    \end{equation}
    Then for any $c \in \mathbb{C}$ with $|\operatorname{Re} c| > R$, $E_c\cap(K\times S)=\varnothing$.
    Thus for every $(z,t) \in E_c$ we have
    \begin{equation}
        \|\nabla_z F(z,t)\| = \|\nabla f_t(z)\| \neq 0.
    \end{equation}
    Here $\nabla_z F$ denotes the gradient of $F$ with respect to the $z$-variables only.
    Hence $\mathrm{d}F(z,t) \neq 0$ on $E_c$, and $E_c$ is a smooth hypersurface of $M \times S$.

    Next, we prove that $\pi : E_c \to S$ is a locally trivial smooth fibration.

    Fix a point
    \begin{equation}
        t_0 = (t_0^1, \ldots, t_0^m) \in S.
    \end{equation}
    Since $S \subset \mathbb{C}^m$ is open, there exist disks
    \begin{equation}
        D_i \subset \mathbb{C}, \quad 1 \le i \le m,
    \end{equation}
    each centered at $t_0^i$, such that
    \begin{equation}
        D_1 \times \cdots \times D_m \subset S.
    \end{equation}
    For $1 \le k \le m$, denote
    \begin{equation}
        U_k = D_1 \times \cdots \times D_k \times \{t_0^{k+1}\} \times \cdots \times \{t_0^m\}.
    \end{equation}
    It suffices to prove that
    \begin{equation}
        \pi^{-1}(U_m) \to U_m
    \end{equation}
    is trivial over $U_m$.
    We argue by induction on $m$; the case $m = 0$ is trivial.
    Assume that for $m - 1$ the projection is trivial over $U_{m-1}$ and that there exists a smooth deformation retraction of $\pi^{-1}(U_{m-1})$ onto $\pi^{-1}(U_0) = \pi^{-1}(t_0)$ whose restriction to $\pi^{-1}(t^1, \ldots, t^{m-1}, t_0^m)$ is an isotopy for every $(t^1, \ldots, t^{m-1}, t_0^m) \in U_{m-1}$.

    On $E_c$, define the vector field
    \begin{equation}
        X_m = \frac{\partial}{\partial t^m} - \frac{g_m(z)}{\|\nabla_z F(z,t)\|^2}\, \overline{\nabla_z F(z,t)}.
    \end{equation}
    We compute
    \begin{equation}
        \begin{aligned}
            X_m F(z,t)
            &= \frac{\partial F}{\partial t^m} - \frac{g_m(z)}{\|\nabla_z F(z,t)\|^2}\, \overline{\nabla_z F(z,t)}\, F \\
            &= g_m(z) - \frac{g_m(z)}{\|\nabla_z F(z,t)\|^2}\, \|\nabla_z F(z,t)\|^2 \\
            &= 0.
        \end{aligned}
    \end{equation}
    Hence $X_m$ is tangent to $E_c$.

    Moreover,
    \begin{equation}
        \pi_* X_m = \frac{\partial}{\partial t^m},
    \end{equation}
    so $X_m$ lifts the constant vector field $\partial/\partial t^m$.

    Take the standard Hermitian metric on $S$.
    By the assumption, for every $(z,t) \in E_c$ we have
    \begin{equation}
        \|X_m\| = \left\| \frac{\partial}{\partial t^m} - \frac{g_m(z)}{\|\nabla_z F(z,t)\|^2}\, \overline{\nabla_z F(z,t)} \right\| \le 1 + \frac{|g_m(z)|}{\|\nabla_z F(z,t)\|} < 1 + \frac{1}{C_m}.
    \end{equation}
    Hence $X_m$ is smooth and bounded on $E_c$.

    By an argument analogous to that in Proposition~\ref{prop:fibration}, there exists a smooth deformation retraction of $\pi^{-1}(U_m)$ onto $\pi^{-1}(U_{m-1})$ whose restriction to $\pi^{-1}(t^1, \ldots, t^{m-1}, t^m)$ is an isotopy for every $(t^1, \ldots, t^{m-1}, t^m) \in U_m$.
    By the induction hypothesis, there also exists a smooth deformation retraction of $\pi^{-1}(U_m)$ onto $\pi^{-1}(t_0)$ whose restriction to $\pi^{-1}(t^1, \ldots, t^m)$ is an isotopy for every $(t^1, \ldots, t^m) \in U_m$.
    This yields a trivialization of $\pi$ over $U_m$.
\end{proof}

We now clarify the relation between regular fibers and sublevel regions of $f$.

For $a > 0$, define
\begin{equation}
    M_{-}^{a} := \{p \in M : \operatorname{Re}(f(p)) \le -a\},
\end{equation}
\begin{equation}
    M_{+}^{a} := \{p \in M : \operatorname{Re}(f(p)) \ge a\}.
\end{equation}

\begin{proposition}\label{prop:homotopy}
    Let $(M, f)$ be a fundamentally tame LG model.
    There exist constants $R_{-} > 0$ and $R_{+} > 0$ such that:

    (1) For any $a, b \ge R_{-}$, the pairs $(M, M_{-}^{a})$ and $(M, M_{-}^{b})$ are homotopy equivalent.
    Moreover, for any $a \ge R_{-}$ and any $c \in \mathbb{C}$ with $\operatorname{Re}(c) \le -R_{-}$, the pair $(M, M_{-}^{a})$ is homotopy equivalent to $(M, f^{-1}(c))$; we denote their common homotopy type by $(M, f^{-\infty})$.

    (2) For any $a, b \ge R_{+}$, the pairs $(M, M_{+}^{a})$ and $(M, M_{+}^{b})$ are homotopy equivalent.
    Moreover, for any $a \ge R_{+}$ and any $c \in \mathbb{C}$ with $\operatorname{Re}(c) \ge R_{+}$, the pair $(M, M_{+}^{a})$ is homotopy equivalent to $(M, f^{-1}(c))$; we denote their common homotopy type by $(M, f^{+\infty})$.
\end{proposition}

\begin{proof}
    We prove the negative case; the positive case is analogous.

    For $a \ge 0$, set
    \begin{equation}
        \mathbb{C}_{-}^{a} := \{c \in \mathbb{C} : \operatorname{Re}(c) \le -a\}.
    \end{equation}
    Then $M_{-}^{a} = f^{-1}(\mathbb{C}_{-}^{a})$.

    Since $f$ is fundamentally tame, there exists $R_{-} > 0$ such that $\mathbb{C}_{-}^{R_{-}}$ contains no critical value of $f$.
    By Theorem~\ref{thm:range}, $\mathbb{C}_{-}^{R_{-}}$ is contained in the range of $f$; by Proposition~\ref{prop:fibration}, the restriction
    \begin{equation}
        f : M_{-}^{R_{-}} \to \mathbb{C}_{-}^{R_{-}}
    \end{equation}
    is a smooth locally trivial fibration.

    For $a \ge R_{-}$, the half-plane $\mathbb{C}_{-}^{a}$ is a deformation retract of $\mathbb{C}_{-}^{R_{-}}$.
    Since the fibration has no critical value over $\mathbb{C}_{-}^{R_{-}}$, this deformation retract lifts to a homotopy equivalence of pairs
    \begin{equation}
        (M, M_{-}^{a}) \simeq (M, M_{-}^{R_{-}}).
    \end{equation}
    Similarly, for any $c$ with $\operatorname{Re}(c) \le -R_{-}$, the singleton $\{c\}$ is a deformation retract of $\mathbb{C}_{-}^{R_{-}}$.
    Hence
    \begin{equation}
        (M, f^{-1}(c)) \simeq (M, M_{-}^{R_{-}}).
    \end{equation}
    This completes the proof.
\end{proof}

We now come to the main result of this section.
Combining the canonical homotopy type at infinity with the deformation fibration, we obtain the following global statement.

\begin{theorem}\label{thm:bundle}
    Assume the hypotheses of Theorem~\ref{thm:deformation} and suppose that $M$ is Stein.
    Let $(M, f_t^{\pm\infty})$ denote the canonical relative homotopy type defined in Proposition~\ref{prop:homotopy}.
    Then the relative cohomology groups
    \begin{equation}
        \mathcal{H}^{\pm} = \coprod_{t\in S} H^n(M, f_t^{\pm\infty}; \mathbb{C})
    \end{equation}
    form a smooth vector bundle over $S$ of rank $\mu$, endowed with a natural flat Gauss--Manin connection.

    Moreover, the total space
    \begin{equation}
        \mathcal{H}_{\mathrm{tot}} = \mathcal{H}^{+} \oplus \mathcal{H}^{-}
    \end{equation}
    is called the cohomological Milnor bundle associated with the deformation $F$.
\end{theorem}

\begin{proof}
    We prove the statement for $\mathcal{H}^{+}$; the case of $\mathcal{H}^{-}$ is analogous.
    By Theorem~\ref{thm:deformation}, for $c \in \mathbb{C}$ with $\operatorname{Re} c$ sufficiently large, the projection
    \begin{equation}
        \pi : E_c \to S
    \end{equation}
    is a locally trivial smooth fibration, where
    \begin{equation}
        E_c = \{(z,t) \in M \times S \mid F(z,t) = c\}.
    \end{equation}

    By the assumption, we may choose the constant $R_+$ of Proposition~\ref{prop:homotopy} uniformly in $t$, and then choose $c$ so that Proposition~\ref{prop:homotopy} applies to $f_t$ for all $t \in S$.
    Hence the pair $(M, f_t^{-1}(c))$ has the homotopy type $(M, f_t^{+\infty})$ for every $t \in S$.

    Fix $t_0 \in S$ and let $U \subset S$ be a contractible open neighborhood of $t_0$.
    By Theorem~\ref{thm:deformation}, there exists a smooth trivialization
    \begin{equation}
        \Phi : \pi^{-1}(U) \longrightarrow \pi^{-1}(t_0) \times U.
    \end{equation}
    For $t_1,t_2 \in U$, this induces a diffeomorphism
    \begin{equation}
        \Phi_{t_1,t_2} : \pi^{-1}(t_1) \longrightarrow \pi^{-1}(t_2),
    \end{equation}
    depending smoothly on $t_1, t_2$.

    Since the projection to $M$
    \begin{equation}
        p_t : \pi^{-1}(t) \longrightarrow f_t^{-1}(c)
    \end{equation}
    is a diffeomorphism for each $t$, we obtain a smooth family of diffeomorphisms
    \begin{equation}
        \Psi_{t_1,t_2} := p_{t_2} \circ \Phi_{t_1,t_2} \circ p_{t_1}^{-1} : f_{t_1}^{-1}(c) \longrightarrow f_{t_2}^{-1}(c).
    \end{equation}

    Because $U$ is contractible, there exists a smooth deformation retraction
    \begin{equation}
        r : U \times [0,1] \longrightarrow U,
    \end{equation}
    with $r(t,0) = t$ and $r(t,1) = t_0$.
    The composition
    \begin{equation}
        H(z,t,\tau) = \Psi_{t, r(t,\tau)}(z)
    \end{equation}
    defines a smooth isotopy in $M$ between the level sets $f_t^{-1}(c)$ and $f_{t_0}^{-1}(c)$ for all $t \in U$.

    This isotopy induces canonical isomorphisms
    \begin{equation}
        H^n(M, f_t^{\pm\infty}; \mathbb{C}) \cong H^n(M, f_{t_0}^{\pm\infty}; \mathbb{C}),
    \end{equation}
    varying smoothly with $t$.
    Hence the relative cohomology groups form a local system over $S$, and therefore determine a smooth vector bundle equipped with the flat Gauss--Manin connection.
    By Theorem~\ref{thm:main}, the rank of this bundle equals $\mu$.
\end{proof}

\begin{remark}\label{rem:picard-lefschetz}
    In classical Picard--Lefschetz theory, when $f$ is a Morse function the Lefschetz thimbles associated with its critical points form a basis of the relative homology
    \begin{equation}
        H_n\bigl(M, \{\operatorname{Re} f \ll 0\}\bigr),
    \end{equation}
    see for instance \cite{milnor1968singular}.
    This description relies on the Morse condition and therefore applies only to those deformation parameters for which the function $f_t$ has nondegenerate critical points.

    In contrast, the construction above does not require $f_t$ to be Morse.
    The relative homotopy type at infinity is identified directly with that of a regular fiber, so that the resulting relative cohomology bundle and its flat connection are defined over the whole deformation parameter space.
\end{remark}

\nocite{CIR2014,FanLanYang2020,Marsh2013TheBC}

\printbibliography

\end{document}